\documentclass[11pt]{amsart}
\usepackage{oldgerm}
\usepackage{amssymb} 
\usepackage{mathrsfs} 
\usepackage{amsmath}
\usepackage{latexsym}
\usepackage[all]{xy}
\usepackage[dvips]{graphics}
\usepackage{amsthm}
\usepackage{enumerate}
\usepackage{bm}
\usepackage{here}
\usepackage{arydshln}

\allowdisplaybreaks

\newcommand{\lt}{{}^\mathrm t\hspace{-0.5mm}}

  \makeatletter
    
   \@addtoreset{table}{section}
  \makeatother

\theoremstyle{plain} 
\newtheorem{thm}{Theorem}[section] 
\newtheorem{lem}[thm]{Lemma} 
\newtheorem{cor}[thm]{Corollary} 
\theoremstyle{definition} 
\newtheorem{rem}[thm]{Remark} 
\theoremstyle{remark} 
\title[An extremal Type II lattice with no $4$-frames]
{A $40$-dimensional extremal Type II lattice with no $4$-frames
}

\author{Norifumi Ojiro
}
\address{Department of Mathematics, 
Graduate School of Science, 
Hiroshima University \\
1-3-1, Kagamiyama, Higashi-Hiroshima, Hiroshima, 739-8526, Japan}
\email{norifumi.ojiro@gmail.com}

\date{}


\begin{document}
\maketitle

\begin{abstract}
We construct a $40$-dimensional extremal Type II lattice not having any subsets consisting of $40$ orthogonal minimal vectors, and determine the automorphism group. This lattice gives an example different from the $16470$ lattices constructed from binary codes by classical constructions.
\end{abstract}

\section{Introduction}
Let $L$ be a lattice, which is a free $\mathbb{Z}$-module of finite rank with a positive definite symmetric bilinear form:
\begin{equation*}
L\times L\ni(\bm{x},\bm{y})\mapsto \bm{x}\cdot\bm{y}\in\mathbb{R}.
\end{equation*}
The lattice $L$ is integral if $\bm{x}\cdot\bm{y}\in\mathbb{Z}$ for all $\bm{x},\bm{y}\in L$, and furthermore $L$ is said to be even if $\bm{x}\cdot\bm{x}\in2\mathbb{Z}$ for all $\bm{x}\in L$. We denote by $L^{\vee}$ the dual lattice of $L$:
\begin{equation*}
L^{\vee}:=\{\bm{x}\in L\otimes\mathbb{R}\ |\ \bm{x}\cdot\bm{y}\in\mathbb{Z}\ {\rm for\ all}\ \bm{y}\in L\}.
\end{equation*}
It is obvious by definition that $L\subset L^{\vee}$ if and only if $L$ is integral. The lattice $L$ is said to be unimodular when $L=L^{\vee}$. An even unimodular lattice is also called a Type II lattice. It is well-known that the rank of a Type II lattice is a multiple of $8$. We denote by ${\rm min}(L)$ the minimal norm of $L$ which is the minimum of norms $\bm{x}\cdot\bm{x}$ for all $\bm{x}\in L\setminus\{\bm{0}\}$. By the theory of modular forms, the minimal norm of an $n$-dimensional Type II lattice $L$ satisfies
\begin{equation*}
{\rm min}(L)\le2(1+\lfloor n/24\rfloor),
\end{equation*} 
and the lattice $L$ is said to be extremal when the equality holds.
We regard an $n$-dimensional lattice $L$ as being in $\mathbb{R}^n$ with the canonical inner product, and write an element of $\mathbb{R}^n$ as a row vector.
Then $\bm{x}\cdot\bm{y}=\bm{x}\,\lt\bm{y}$ for $\bm{x}$,$\bm{y}\in\mathbb{R}^n$.
A $4$-frame is a subset of $\mathbb{R}^{n}$ consisting of $n$ orthogonal $4$-vectors:
\begin{equation*}
\bm{v}_i\,\lt\bm{v}_j=
\begin{cases}
4&{\rm for}\ 1\le i=j\le n,\\0&{\rm for}\ 1\le i\neq j\le n.
\end{cases}
\end{equation*}

According to Theorem $1$ of \cite{kkm}, an even lattice with a $4$-frame is constructed from a doubly even binary code, and
by Theorem $3$ of \cite{kkm} we see that an extremal doubly even self-dual binary code of length $40$ corresponds bijectively to a $40$-dimensional extremal Type II lattice with a $4$-frame.
Then it was shown in \cite{bhm} that the number of extremal doubly even self-dual binary codes of length $40$ is $16470$. Further by \cite{h} all extremal doubly even self-dual binary codes of length $40$ are obtained as the residue codes of some extremal Type II codes over $\mathbb{Z}/4\mathbb{Z}$ of the same length. As a result, we know that there are exactly $16470$ extremal Type II lattices of rank $40$ with a $4$-frame and they are constructed from extremal doubly even self-dual binary codes or some extremal Type II codes over $\mathbb{Z}/4\mathbb{Z}$ of length $40$.
On the other hand, by \cite{ps} there are more than $8.45\times10^{51}$ extremal Type II lattices of rank $40$, whereas so many examples do not seem to be known: Mckay's $M_{40}$, the lattice
$(U_5(2)\times2^{1+4}_{-}.Alt_5).2$
in the website \cite{ns} (we simply denote this lattice by $N_{40}$ in this paper), and furthermore the $16470$ lattices with a $4$-frame e.g. three lattices $\mathfrak{L}_1$, $\mathfrak{L}_2$, $\mathfrak{L}_3$ in \cite{o1}, the lattice constructed in \cite{cas} from a double circulant code over $\mathbb{Z}/4\mathbb{Z}$. Note that perhaps $M_{40}$, $N_{40}$ may belong to the $16470$ examples.

Recently in \cite{s} an extremal Type II lattice of rank $64$ was constructed from a generalized quadratic residue code of length $32$ and the discriminant group of an even lattice of rank $2$. In the same way, we shall construct a certain $40$-dimensional lattice denoted by $O_{40}$ as an overlattice of the lattice with a Gram matrix
\begin{equation}\label{l}
\begin{pmatrix}
6I_{20}&3I_{20}\\
3I_{20}&12I_{20}
\end{pmatrix},
\end{equation}
where $I_{n}$ denotes the unit matrix of rank $n$.
Our results are as follows:
\begin{thm}\label{t1}
The lattice $O_{40}$ is a $40$-dimensional extremal Type II lattice with no $4$-frames.
\end{thm}
\begin{thm}\label{t2}
The automorphism group of $O_{40}$ is of order $684$, which is isomorphic to the semidirect product $\mathbb{Z}/19\mathbb{Z}\rtimes_{\psi}\mathbb{Z}/36\mathbb{Z}$ of additive groups $\mathbb{Z}/19\mathbb{Z}$ and $\mathbb{Z}/36\mathbb{Z}$, where $\psi$ is the group homomorphism defined by
\begin{equation*}
\psi:\mathbb{Z}/36\mathbb{Z}\ni b\mapsto(c\mapsto3^bc)\in{\rm Aut}(\mathbb{Z}/19\mathbb{Z}).
\end{equation*}
\end{thm}
The automorphism group of $M_{40}$ contains a subgroup of order $13680$ (cf. \cite[\S$5$ of Chapter $8$]{cs}). The automorphism group of $N_{40}$ is of order $2^{18}\times3^{6}\times5^{2}\times11$, see \cite{ns}.
Hence we have the following corollary:
\begin{cor}\label{c1}
The lattice $O_{40}$ is not isometric to any lattice of $M_{40}$, $N_{40}$ and the $16470$ extremal Type II lattices with a $4$-frame.
\end{cor}
In particular Corollary $\ref{c1}$ implies that $O_{40}$ cannot be obtained from binary and some $\mathbb{Z}_4$-codes by classical constructions, as mentioned above. We also have the following by examining the $4$-vectors of $O_{40}$:
\begin{lem}\label{p1}
The lattice $O_{40}$ contains
\begin{equation*}
M:=A_1(2)\oplus A_1(2)\oplus A_{19}(2)\oplus A_{19}(2)
\end{equation*}
as a sublattice, where $A_{n}(2)$ denotes the lattice arising from scaling up the irreducible root lattice $A_{n}$ by $\sqrt{2}$ and $\oplus$ denotes orthogonal sum.
\end{lem}
Let $L$ be the lattice with a Gram matrix given by \eqref{l}.
We denote by $L+M$ the lattice consisting of the vectors $\bm{v}+\bm{u}$ for all $\bm{v}\in L$, $\bm{u}\in M$. Then one has the following:
\begin{thm}\label{t3}
$O_{40}=L+M$.
\end{thm}
This implies that $O_{40}$ can be obtained by gluing $M$ to $L$.

In the next section we construct $O_{40}$ from two generalized quadratic residue codes of length $20$. In Section $3$ we investigate the types of the $4$-vectors of $O_{40}$ and give a proof of Theorem \ref{t3}. In Section $4$ we show that $O_{40}$ does not contain any $4$-frames, proving that $O_{40}$ contains at most $32$ orthogonal $4$-vectors (see Lemma \ref{l1}). In Section $5$ we determine the automorphism group of $O_{40}$ by utilizing that $O_{40}$ has a basis consisting of $4$-vectors. A Gram matrix of $O_{40}$ is presented as an appendix in Section $6$. Results in this paper are based on calculations done by using {\tt GAP} \cite{g} and {\tt PARI/GP} \cite{p}.
Those computational data files are available upon request addressed to the author.

The author would like to express his gratitude to Professor Ichiro Shimada who informed him of References \cite{bhm}, \cite{h}, \cite{kkm} and \cite{s},
also advised him on this manuscript.

\section{Construction}
We construct a $40$-dimensional extremal Type II lattice as an even overlattice of an even lattice using two generalized quadratic residue codes.
As references on this section, the readers are recommended to see \cite[in particular, Chapter $3.3$]{e} and \cite{s}.

Let $R$ be the $2$-dimensional even lattice with a Gram matrix
\begin{equation*}
\begin{pmatrix}
6&3\\
3&12
\end{pmatrix},
\end{equation*}
and let $L$ be the $40$-dimensional even lattice which is the orthogonal sum of $20$ copies of $R$. Then a Gram matrix of $L$ is written as in \eqref{l}.
The discriminant group of $L$, namely $D_L:=L^{\vee}/L$ is an abelian group of order $63^{20}$ and which has a natural bilinear form
\begin{equation*}
b_L(\bm{u}+L,\bm{v}+L):=\bm{u}\,\lt\bm{v}+\mathbb{Z}\in\mathbb{Q}/\mathbb{Z}
\end{equation*} 
and a quadratic form 
\begin{equation*}
q_L(\bm{u}+L):=\bm{u}\,\lt\bm{u}+2\mathbb{Z}\in\mathbb{Q}/2\mathbb{Z}.
\end{equation*}
Let $\{\bm{e}_1,\dots,\bm{e}_{40}\}$ be a basis  of $L^{\vee}$ with a Gram matrix
\begin{equation*}
\frac{1}{21}
\begin{pmatrix}
28I_{20}&7I_{20}\\
7I_{20}&2I_{20}
\end{pmatrix}.
\end{equation*}
Then $D_L$ can be written as
\begin{equation*}
\mathbb{Z}/3\mathbb{Z}\ \bm{e}_{1}+\dots+\mathbb{Z}/3\mathbb{Z}\ \bm{e}_{20}+\mathbb{Z}/21\mathbb{Z}\ \bm{e}_{21}+\dots+\mathbb{Z}/21\mathbb{Z}\ \bm{e}_{40}.
\end{equation*}
We consider a subgroup $C$ of $D_L$ satisfying the conditions
\begin{enumerate}
\item[(i)] $C$ is totally isotropic, i.e. for all $\bm{u}+L,\bm{v}+L\in C$,\\
\hspace{10mm}$b_L(\bm{u}+L,\bm{v}+L)=\mathbb{Z}$ and $q_L(\bm{u}+L)=2\mathbb{Z}$.
\item[(ii)] $|C|=63^{10}$, where $|C|$ denotes the cardinality of $C$.
\end{enumerate}
For $p:L^{\vee}\rightarrow D_L$, put $L_C:=p^{-1}(C)$. Then $L\subset L_C\subset L^\vee$. It is easy to see that $L_C$ is even by (i) and also unimodular by (ii).
We will give $C$ by utilizing a direct product of two generalized quadratic residue codes of length $20$ over $\mathbb{Z}/3\mathbb{Z}$ and $\mathbb{Z}/21\mathbb{Z}$.

Let $A$ be a commutative ring with unity and let $A^{n}$ be the free $A$-module
which is the direct product of $n$ copies of $A$. We give a numbering $(1,2,3,\dots,18,0,\infty)$ to the coordinates of $A^{20}$, and identify a finite field $\mathbb{F}_{19}$ with $\{1,2,\dots,18,0\}$.
Then a generalized quadratic residue code of length $20$ over $A$ with a parameter $(a,b,d,s,t,e)\in A^{6}$ is a $A$-submodule of $A^{20}$ generated by the vectors $\bm{u}_{1},\dots,\bm{u}_{18},\bm{u}_{0},\bm{u}_{\infty}$ which are defined as follows:
\begin{eqnarray*}
\bm{u}_{\infty}(j)&:=&
\begin{cases}
a&{\rm for}\ j\in\mathbb{F}_{19},\\
b&{\rm for}\ j=\infty,
\end{cases}
\\
\bm{u}_{i}(j)&:=&
\begin{cases}
d&{\rm for}\ i,j\in\mathbb{F}_{19}\ {\rm with}\ i=j,\\
s&{\rm for}\ i,j\in\mathbb{F}_{19}\ {\rm with}\ \chi_{19}(i-j)=1,\\
t&{\rm for}\ i,j\in\mathbb{F}_{19}\ {\rm with}\ \chi_{19}(i-j)=-1,\\
e&{\rm for}\ i\in\mathbb{F}_{19}\ {\rm and}\ j=\infty,
\end{cases}
\end{eqnarray*}
where $\chi_{19}$ denotes the Legendre character modulo $19$.

Let $\mathscr{C}_{n}$ be a generalized quadratic residue code of length $20$ over $\mathbb{Z}/n\mathbb{Z}$ with some parameter, and let
$\bm{u}_{i}$ and $\bm{w}_{i}$ be the row vectors of $20$ components which are the generators of $\mathscr{C}_{3}$ and $\mathscr{C}_{21}$ respectively, defined as above. We denote by $C_3\times C_{21}$ the subgroup of $D_L$ generated by the direct product $\mathscr{C}_{3}\times\mathscr{C}_{21}\subset(\mathbb{Z}/3\mathbb{Z})^{20}\times(\mathbb{Z}/21\mathbb{Z})^{20}$.
Then one can replace (i) by
\begin{equation}
\tag{i'}
(\bm{u}_{i},\bm{w}_{i})
\begin{pmatrix}
28I_{20}&7I_{20}\\
7I_{20}&2I_{20}
\end{pmatrix}
\lt(\bm{u}_{j},\bm{w}_{j})
\equiv0\ 
\begin{cases}
{\rm mod}\,(21)&(i\neq j)\\
{\rm mod}\,(42)&(i=j).
\end{cases}
\end{equation}
Further (ii) is replaced by
\begin{enumerate}
\item[(ii')] For a basis $\{\bm{v}_1,\dots,\bm{v}_{40}\}$ of $L_{C_{3}\times C_{21}}$,
let $B$ be a $40$-by-$40$ integral matrix such that
\begin{equation*}
\begin{pmatrix}
\bm{v}_1\\\vdots\\\bm{v}_{40}
\end{pmatrix}
=
B
\begin{pmatrix}
\bm{e}_1\\\vdots\\\bm{e}_{40}
\end{pmatrix}.
\end{equation*}
Then $\det(B)=63^{10}$, where $\det(B)$ is the determinant of $B$.
\end{enumerate}
We search a code $\mathscr{C}_{3}\times\mathscr{C}_{21}$ satisfying (i') and (ii').
Exploring by a computer one finds the code satisfying the conditions, which has a pair of parameters
\begin{equation*}
(1, 0, 0, 0, 1, 1)\in{(\mathbb{Z}/3\mathbb{Z})}^6\ {\rm and}\ (0, 7, 1, 0, 4, 17)\in{(\mathbb{Z}/21\mathbb{Z})}^6.
\end{equation*}
Then the matrix $B$ can be written as the form
\begin{equation}\label{basis}
B:=
\begin{pmatrix}
I_{20}&B_1\\
O&B_2
\end{pmatrix},
\ \ 
B_2:=
\begin{pmatrix}
3I_{10}&B_3\\
O&21I_{10}
\end{pmatrix},
\end{equation}
where $B_1$ and $B_3$ are $20$-by-$20$ and $10$-by-$10$ matrices respectively as follows:

\setcounter{MaxMatrixCols}{40}
\renewcommand{\arraystretch}{0.5}
{\small
\begin{equation*}
{\tabcolsep=0.5mm
\left(
\begin{tabular}
{cccccccccccccccccccc}
0&2&1&1&2&2&2&2&1&2&4&20&16&10&4&10&2&8&16&7 \\
   1&0&2&1&1&2&2&2&2&1&17&10&11&19&4&7&10&11&20&13 \\
   2&1&0&2&1&1&2&2&2&2&19&20&16&2&16&16&7&4&17&16 \\
   2&2&1&0&2&1&1&2&2&2&17&4&11&19&17&19&16&16&16&1 \\
   1&2&2&1&0&2&1&1&2&2&17&2&16&14&13&20&19&4&7&7 \\
   1&1&2&2&1&0&2&1&1&2&8&11&11&13&20&10&20&10&13&1 \\
   1&1&1&2&2&1&0&2&1&1&14&17&8&5&4&14&10&2&7&10 \\
   1&1&1&1&2&2&1&0&2&1&7&8&5&5&11&1&14&1&11&4 \\
   2&1&1&1&1&2&2&1&0&2&10&19&11&11&14&17&1&11&4&16 \\
   1&2&1&1&1&1&2&2&1&0&5&7&13&20&14&2&17&7&5&13 \\
   2&1&2&1&1&1&1&2&2&1&3&2&1&1&2&2&2&2&1&10 \\
   1&2&1&2&1&1&1&1&2&2&1&3&2&1&1&2&2&2&2&10 \\
   2&1&2&1&2&1&1&1&1&2&2&1&3&2&1&1&2&2&2&10 \\
   2&2&1&2&1&2&1&1&1&1&2&2&1&3&2&1&1&2&2&10 \\
   2&2&2&1&2&1&2&1&1&1&1&2&2&1&3&2&1&1&2&10 \\
   2&2&2&2&1&2&1&2&1&1&1&1&2&2&1&3&2&1&1&10 \\
   1&2&2&2&2&1&2&1&2&1&1&1&1&2&2&1&3&2&1&10 \\
   1&1&2&2&2&2&1&2&1&2&1&1&1&1&2&2&1&3&2&10 \\
   2&1&1&2&2&2&2&1&2&1&2&1&1&1&1&2&2&1&3&10 \\
   1&1&1&1&1&1&1&1&1&1&1&1&1&1&1&1&1&1&1&11  
\end{tabular}
\right),\ 
\left(
\begin{tabular}
{cccccccccc}
18&3&6&12&18&12&0&15&6&3\\
 6&12&12&3&18&15&12&12&3&18\\
 3&3&6&0&6&6&15&18&6&15\\
 6&18&12&3&6&3&6&6&6&9\\
 6&0&6&9&9&3&3&18&15&3\\
 15&12&12&9&3&12&3&12&9&9\\
  9&6&15&18&18&9&12&0&15&0\\
  15&15&18&18&12&0&9&0&12&6\\
  12&3&12&12&9&6&0&12&18&15\\
  18&15&9&3&9&0&6&15&18&18
\end{tabular}
\right).
}
\end{equation*}
}
Applying for example {\tt ShortestVectors} of {\tt GAP}
to $L_{C_3\times C_{21}}$,
one can confirm within one minute that $L_{C_3\times C_{21}}$ is extremal, i.e. ${\rm min}(L_{C_3\times C_{21}})=4$.
We denote $L_{C_3\times C_{21}}$ by $O_{40}$.
\begin{rem}\label{r1}
Since a generalized quadratic residue code of length $20$ is invariant under the permutations of the coordinates $p_c:x_{t}\mapsto x_{t+1}$ of order $19$ and $p_s:x_{t}\mapsto x_{4t}$ of order $9$, the automorphism group ${\rm Aut}(O_{40})$ has a subgroup $\varGamma$ of order $342$ which arises from $p_c$, $p_s$ and the change of sign (cf. \cite[Section $3$]{s}).
\end{rem}

\section{Types of $4$-vectors}
Let $S$ be the set of the $4$-vectors of $O_{40}$. By the theory of modular forms (cf. e.g. \cite[Chapter $7$]{se}),
one has $|S|=39600$. It follows immediately from ${\rm min}(O_{40})=4$ that
\begin{equation*}
\bm{u}\,\lt\bm{v}\in\{0,\pm1,\pm2,\pm4\}\ \ {\rm for}\ {\rm all}\ \bm{u},\bm{v}\in S.
\end{equation*}
The type of $\bm{v}\in S$ is defined by
\begin{equation*}
t(\bm{v}):=[\,t_0(\bm{v}),\ t_1(\bm{v}),\ t_2(\bm{v}),\ t_4(\bm{v})\,],\ \ t_n(\bm{v}):=\big{|}\{\bm{u}\in S\ |\ \bm{u}\,\lt\bm{v}=n\}\big{|}.
\end{equation*}
By definition $t_0(\bm{v})+2(t_1(\bm{v})+t_2(\bm{v})+t_4(\bm{v}))=39600$ and $t_4(\bm{v})=1$ for all $\bm{v}\in S$. Using the function {\tt qfminim} of {\tt PARI/GP}, we obtain the list of the elements of $S$.
It turns out by direct calculation that $18$ types appear for the vectors of $S$,
and then we divide $S$ into subsets $S_1,\dots,S_{18}$ which each consist of vectors with the same type, see Table \ref{type18}.

\renewcommand{\arraystretch}{0.9}
{\small
\begin{table}
\begin{tabular}
{c|cccc|c}
$S_i$&$t_{0}$&$t_{1}$&$t_{2}$&$t_{4}$&$|S_i|$\\ \hline
$S_1$&24018&7752&38&1&4\\
$S_2$&24234&7608&74&1&304\\
$S_3$&24270&7584&80&1&684\\
$S_4$&24342&7536&92&1&912\\
$S_5$&24378&7512&98&1&2736\\
$S_6$&24414&7488&104&1&5472\\
$S_7$&24450&7464&110&1&9576\\
$S_8$&24486&7440&116&1&4788\\
$S_9$&24522&7416&122&1&4788\\
$S_{10}$&24558&7392&128&1&5016\\
$S_{11}$&24594&7368&134&1&1368\\
$S_{12}$&24630&7344&140&1&2052\\
$S_{13}$&24666&7320&146&1&228\\
$S_{14}$&24738&7272&158&1&684\\
$S_{15}$&24810&7224&170&1&152\\
$S_{16}$&24918&7152&188&1&684\\
$S_{17}$&24990&7104&200&1&76\\
$S_{18}$&25206&6960&236&1&76  
\end{tabular}
\vspace{3mm}
\caption{Division of $S$ by the type.}
\label{type18}
\end{table}
}
Similarly, for $1\le i\le 18$ the $S_i$-type of $\bm{v}\in S_i$ is defined by
\begin{equation*}
t^i(\bm{v}):=[\,t^i_0(\bm{v}),\ t^i_1(\bm{v}),\ t^i_2(\bm{v}),\ t^i_4(\bm{v})\,],\ \ t^i_n(\bm{v}):=\big{|}\{\bm{u}\in S_i\ |\ \bm{u}\,\lt\bm{v}=n\}\big{|}.
\end{equation*}
Actually the vectors of $S_i$ have $m_i$ possible $S_i$-types where $m_i\ge1$, and we divide $S_i$ into subsets $S_{i,1},\dots,S_{i,m_i}$ which each consist of vectors with the same $S_i$-type.

Further, for $1\le j\le m_i$ we compute the $S_{i,j}$-type of $\bm{v}\in S_{i,j}$ defined by
\begin{equation*}
t^{i,j}(\bm{v}):=[\,t^{i,j}_0(\bm{v}),\ t^{i,j}_1(\bm{v}),\ t^{i,j}_2(\bm{v}),\ t^{i,j}_4(\bm{v})\,],\ \ t^{i,j}_n(\bm{v}):=\big{|}\{\bm{u}\in S_{i,j}\ |\ \bm{u}\,\lt\bm{v}=n\}\big{|}.
\end{equation*}
It turns out that the vectors of $S_{i,j}$ have the only $S_{i,j}$-type,
see Table \ref{type64irr}.

\setcounter{MaxMatrixCols}{40}
\renewcommand{\arraystretch}{0.5}
{\small
\begin{table}
\begin{center}
\begin{tabular}{cc}
\begin{minipage}{0.5\hsize}
\begin{center}
\begin{tabular}
{c|cccc|c}
$S_{i,j}$&$t^{i,j}_{0}$&$t^{i,j}_{1}$&$t^{i,j}_{2}$&$t^{i,j}_{4}$&$|S_{i,j}|$\\ \hline
$S_{1,1}$&2&0&0&1&4 \\ $S_{2,1}$&38&18&0&1&76 \\ $S_{2,2}$&162&32&0&1&228 \\ $S_{3,1}$&434&120&4&1&684 \\ $S_{4,1}$&114&56&0&1&228 \\ $S_{4,2}$&422&128&2&1&684 \\ $S_{5,1}$&362&154&6&1&684 \\ $S_{5,2}$&402&138&2&1&684 \\ $S_{5,3}$&414&130&4&1&684 \\ $S_{5,4}$&438&120&2&1&684 \\ $S_{6,1}$&330&168&8&1&684 \\ $S_{6,2}$&422&128&2&1&684 \\ $S_{6,3}$&422&128&2&1&684 \\ $S_{6,4}$&422&128&2&1&684 \\ $S_{6,5}$&446&114&4&1&684 \\ $S_{6,6}$&434&120&4&1&684 \\ $S_{6,7}$&434&122&2&1&684 \\ $S_{6,8}$&446&112&6&1&684 \\ $S_{7,1}$&330&168&8&1&684 \\ $S_{7,2}$&302&178&12&1&684 \\ $S_{7,3}$&362&160&0&1&684 \\ $S_{7,4}$&374&152&2&1&684 \\ $S_{7,5}$&374&152&2&1&684 \\ $S_{7,6}$&438&120&2&1&684 \\ $S_{7,7}$&850&252&6&1&1368 \\ $S_{7,8}$&426&128&0&1&684 \\ $S_{7,9}$&840&260&3&1&1368 \\ $S_{7,10}$&434&122&2&1&684 \\ $S_{7,11}$&402&138&2&1&684 \\ $S_{7,12}$&390&146&0&1&684 \\ $S_{8,1}$&362&154&6&1&684 \\ $S_{8,2}$&446&114&4&1&684
\end{tabular}
\end{center}
\end{minipage}
\hfill
\begin{minipage}{0.5\hsize}
\begin{center}
\begin{tabular}
{c|cccc|c}
$S_{i,j}$&$t^{i,j}_{0}$&$t^{i,j}_{1}$&$t^{i,j}_{2}$&$t^{i,j}_{4}$&$|S_{i,j}|$\\ \hline
$S_{8,3}$&434&120&4&1&684 \\ $S_{8,4}$&446&114&4&1&684 \\ $S_{8,5}$&434&120&4&1&684 \\ $S_{8,6}$&398&136&6&1&684 \\ $S_{8,7}$&446&112&6&1&684 \\ $S_{9,1}$&302&178&12&1&684 \\ $S_{9,2}$&422&130&0&1&684 \\ $S_{9,3}$&470&98&8&1&684 \\ $S_{9,4}$&434&122&2&1&684 \\ $S_{9,5}$&386&144&4&1&684 \\ $S_{9,6}$&426&128&0&1&684 \\ $S_{9,7}$&450&112&4&1&684 \\ $S_{10,1}$&434&120&4&1&684 \\ $S_{10,2}$&422&128&2&1&684 \\ $S_{10,3}$&422&128&2&1&684 \\ $S_{10,4}$&446&114&4&1&684 \\ $S_{10,5}$&422&128&2&1&684 \\ $S_{10,6}$&398&136&6&1&684 \\ $S_{10,7}$&114&56&0&1&228 \\ $S_{10,8}$&362&154&6&1&684 \\ $S_{11,1}$&422&130&0&1&684 \\ $S_{11,2}$&434&122&2&1&684 \\ $S_{12,1}$&330&168&8&1&684 \\ $S_{12,2}$&446&114&4&1&684 \\ $S_{12,3}$&422&128&2&1&684 \\ $S_{13,1}$&90&62&6&1&228 \\ $S_{14,1}$&450&112&4&1&684 \\ $S_{15,1}$&38&0&18&1&76 \\ $S_{15,2}$&38&18&0&1&76 \\ $S_{16,1}$&614&0&34&1&684 \\ $S_{17,1}$&38&0&18&1&76 \\ $S_{18,1}$&38&0&18&1&76
\end{tabular}
\end{center}
\end{minipage}
\end{tabular}
\vspace{3mm}
\caption{Irreducible subsets and types of $S$.}
\label{type64irr}
\end{center}
\end{table}
}
This implies that $S$ is ultimately divided into $64$ subsets $S_{i,j}$'s by the types. We call $S_{i,j}$ an {\it irreducible subset} and $t^{i,j}$ an {\it irreducible type} from the fact. If $S$ has an irreducible subset with an irreducible type $[\,78,0,0,1\,]$ then $O_{40}$ contains a $4$-frame. But we do not know whether the converse holds true. Therefore we count directly the numbers of orthogonal $4$-vectors contained in $O_{40}$ in the next section.

Here $40$ vectors chosen appropriately from $S_{1,1}\cup S_{18,1}$ form a basis of the lattice $M$ described in Lemma \ref{p1}, and thus $M\subset O_{40}$.
\begin{proof}[Proof of Theorem $\ref{t3}$]
Let $\{\bm{m}_{1},\dots,\bm{m}_{40}\}$ be a basis of $M$ and let $B_M$ be the matrix $B_M$ such that
\begin{equation*}
\begin{pmatrix}
\bm{m}_{1}\\
\vdots\\
\bm{m}_{40}
\end{pmatrix}
=B_MB
\begin{pmatrix}
\bm{e}_1\\
\vdots\\
\bm{e}_{40}
\end{pmatrix},
\end{equation*}
where $B$ is the matrix given by \eqref{basis}. Further we take a basis $\bm{l}_1,\dots,\bm{l}_{40}$ of $L$ such that
\begin{equation*}
\begin{pmatrix}
\bm{l}_{1}\\
\vdots\\
\bm{l}_{40}
\end{pmatrix}
=
\begin{pmatrix}
3I_{20}&O\vspace{2mm}\\
O&21I_{20}
\end{pmatrix}
\begin{pmatrix}
\bm{e}_1\\
\vdots\\
\bm{e}_{40}
\end{pmatrix}.
\end{equation*}
Then by computing one can easily verify that there is a unimodular matrix $U$ of rank $80$ such that
\begin{equation*}
U
\left(
\begin{array}{cc}
B_M\hspace{-7mm}&\hspace{-7mm}B\vspace{2mm}\\\hdashline\\
3I_{20}&O\vspace{2mm}\\
O&21I_{20}
\end{array}
\right)
=
\begin{pmatrix}
\hspace{3.5mm}B&\vspace{1.5mm}\\\hdashline\\\\
\hspace{3.5mm}O&\vspace{1mm}\\
&
\end{pmatrix}.
\end{equation*}
This means that $U\,\lt(\lt\bm{m}_1,\dots,\lt\bm{m}_{40},\lt\bm{l}_1,\dots,\lt\bm{l}_{40})=\lt(\lt\bm{v}_1,\dots,\lt\bm{v}_{40},\lt\bm{0},\dots,\lt\bm{0})$, where $\{\bm{v}_1,\dots,\bm{v}_{40}\}$ is a basis of $O_{40}$. This completes the proof.
\end{proof}
\begin{rem}
From the case where $R$ is of the lattices with Gram matrices
\begin{equation*}
\begin{pmatrix}
4&0\\
0&4
\end{pmatrix}
,
\begin{pmatrix}
4&1\\
1&6
\end{pmatrix}
,
\begin{pmatrix}
4&2\\
2&6
\end{pmatrix}
,
\begin{pmatrix}
6&1\\
1&6
\end{pmatrix}
,
\end{equation*}
we can actually obtain at least two extremal Type II lattices of rank $40$. All those lattices, however, contain a $4$-frame because the lattices have an irreducible subset with an irreducible type $[\,78,0,0,1\,]$, and thus belong to the $16470$ examples. One of those lattices has the $4$-vectors of three types and the automorphism group of order at least $27360$.
\end{rem}

\section{Orthogonal $4$-vectors}
Let $n(O_{40})$ be the maximum of the numbers of orthogonal $4$-vectors contained in $O_{40}$.
Then we prove the following lemma:
\begin{lem}\label{l1}
$n(O_{40})=32$.
\end{lem}
\begin{proof}
For each $\bm{v}\in S$ let $S(\bm{v})$ be a set of orthogonal $4$-vectors containing $\bm{v}$ such that $|S(\bm{v})|$ is largest. By definition it is obvious that
\begin{equation*}
n(O_{40})={\rm max}_{\bm{v}\in S}(|S(\bm{v})|).
\end{equation*}
Since every automorphism $g$ of $O_{40}$ acts on $S$ and $|S(\bm{v}g)|=|S(\bm{v})|$, one has 
\begin{equation*}
{\rm max}_{\bm{v}\in S}(|S(\bm{v})|)={\rm max}_{\bm{v}\in S/\varGamma}(|S(\bm{v})|),
\end{equation*}
where $\varGamma$ is the subgroup of ${\rm Aut}(O_{40})$ described in Remark \ref{r1}.
By computing we have $|S/\varGamma|=132$. Then $S(\bm{v})$ is calculated as follows. For every $\bm{v}\in S$ we compute the set $S_1(\bm{v})$ consisting of $\bm{v}$ and all $4$-vectors perpendicular to $\bm{v}$. We put $S(\bm{v}):=S_1(\bm{v})$ if $|S_1(\bm{v})|=1$, otherwise
put
\begin{equation*}
S_2(\bm{v}):=S_1(\bm{v})\cap S_1(\bm{u}),\ {\rm for}\ \bm{u}\in S_1(\bm{v})
\end{equation*}
such that $|S_1(\bm{v})\cap S_1(\bm{u})|$ is largest. Then $S_2(\bm{v})$ is the set consisting of $2$ orthogonal vectors $\bm{v}$,$\bm{u}$ and all $4$-vectors perpendicular to these $2$ vectors, and by definition
\begin{equation*}
|S_1(\bm{v})|\ge|S_2(\bm{v})|\ge|S_1(\bm{v})\cap S_1(\bm{u'})|\ge2\ {\rm for\ all}\ \bm{u'}\in S_1(\bm{v}).
\end{equation*}
If $|S_2(\bm{v})|=2$ then we put $S(\bm{v}):=S_2(\bm{v})$. If not, repeating this procedure while $|S_n(\bm{v})|>n$ one has $|S_{m_{\bm{v}}}(\bm{v})|={m_{\bm{v}}}$ for some ${m_{\bm{v}}}\le40$, and then put $S(\bm{v}):=S_{m_{\bm{v}}}(\bm{v})$. In Table \ref{distS} we give the number $m_{\bm{v}}$ for $\bm{v}\in S/\varGamma$ and the number $n_{\bm{v}}$ of $\bm{v}$ such that $|S(\bm{v})|=m_{\bm{v}}$.

\setcounter{MaxMatrixCols}{40}
\renewcommand{\arraystretch}{1.5}
\begin{table}[H]
\begin{center}
\begin{tabular}
{c|cccccccccc|c}
$m_{\bm{v}}$&19&20&21&22&23&24&25&26&28&32&total\\ \hline
$n_{\bm{v}}$&6&15&30&12&12&12&15&8&21&1&132
\end{tabular}
\end{center}
\vspace{-3mm}
\caption{Orthogonal $4$-vectors of $O_{40}$.}
\label{distS}
\end{table}
One has ${\rm max}_{\bm{v}\in S/\varGamma}(|S(\bm{v})|)=32$ by Table \ref{distS}, and thus $n(O_{40})=32$.
\end{proof}
By Lemma \ref{l1}, $O_{40}$ does not contain any $4$-frames, and thus a proof of Theorem \ref{t1} has been completed.

\section{The automorphism group}
A $40$-dimensional extremal Type II lattice $L_{40}$ has a basis consisting of $4$-vectors and $6$-vectors, while there is the case where $L_{40}$ cannot have a basis consisting of only $4$-vectors, see Theorem 1 and Remark 1 in \cite{o2}.
Fortunately $O_{40}$ is not of the case: 

\begin{lem}\label{l2}
The lattice $O_{40}$ has a basis consisting of $4$-vectors.
\end{lem}
\begin{proof}
Let $\{\bm{v}_1,\dots,\bm{v}_{40}\}$ be the basis of $O_{40}$ given by the matrix $B$ of \eqref{basis} and let $G(\bm{v}_1,\dots,\bm{v}_{40})$ be the Gram matrix with respect to $\{\bm{v}_1,\dots,\bm{v}_{40}\}$. Applying the function {\tt LLLReducedGramMat} of {\tt GAP} to $G(\bm{v}_1,\dots,\bm{v}_{40})^{-1}$, one has a basis consisting of $34$ vectors of norm $4$ and $6$ vectors of norm $6$. Then one can easily find by computing a basis consisting of $4$-vectors containing those $34$ vectors.
\end{proof}
Actually we can take a basis $\{\bm{v}'_{1},\dots,\bm{v}'_{40}\}$ of $O_{40}$ such that
\begin{eqnarray*}
&\bm{v}'_{1}\in S_{1,1},\  \bm{v}'_{2}\in S_{7,7},\ \bm{v}'_{3},\bm{v}'_{4}\in S_{15,1},\ \bm{v}'_{5}\in S_{2,2},\ \bm{v}'_{18}\in S_{10,3},\ \bm{v}'_{28}\in S_{6,4},&\\ &\bm{v}'_{31}\in S_{10,8},\ \bm{v}'_{34}\in S_{8,4},\ \bm{v}'_{35}\in S_{10,6},\ \bm{v}'_{39}\in S_{9,5},\ \bm{v}'_{40}\in S_{7,4},&\\
&\bm{v}'_{6},\bm{v}'_{8},\bm{v}'_{13},\bm{v}'_{14},\bm{v}'_{17},\bm{v}'_{23},\bm{v}'_{25},\bm{v}'_{26},\bm{v}'_{27},\bm{v}'_{29},\bm{v}'_{30},\bm{v}'_{33},\bm{v}'_{37},\bm{v}'_{38}\in S_{16,1},&\\
&{\rm others}\in S_{9,3}.&
\end{eqnarray*}
The Gram matrix with respect to $\{\bm{v}'_{1},\dots,\bm{v}'_{40}\}$ is given in Section $6$.

Let $B'$ be the matrix such that
\begin{equation*}
B'\,\lt(\lt\bm{v}_{1},\dots,\lt\bm{v}_{40})=\lt(\lt\bm{v}'_{1},\dots,\lt\bm{v}'_{40}).
\end{equation*}
Then one has
\begin{equation*}
B'\,G(\bm{v}_1,\dots,\bm{v}_{40})\,\lt B'=G(\bm{v}'_1,\dots,\bm{v}'_{40}).
\end{equation*}
Since each automorphism of $O_{40}$ is regarded as an element $g$ of ${\rm GL}_{40}(\mathbb{Z})$ fixing $G(\bm{v}_1,\dots,\bm{v}_{40})$, one has
\begin{equation*}
B'g\,G(\bm{v}_1,\dots,\bm{v}_{40})\,\lt(B'g)=B'G(\bm{v}_1,\dots,\bm{v}_{40})\lt B'.
\end{equation*}
Note that $g$ preserves the type of $S$.
Hence putting $B':=\lt(\lt\bm{b}'_{1},\dots,\lt\bm{b}'_{40})$ and $V:=\lt(\lt\bm{v}_{1},\dots,\lt\bm{v}_{40})$, we see that $\bm{b}'_1gV,\dots,\bm{b}'_{40}gV$ each belong to the same irreducible subset with $\bm{v}'_1,\dots,\bm{v}'_{40}$. Then we enumerate all the matrices $X:=\lt(\lt\bm{x}_1,\dots,\lt\bm{x}_{40})$ satisfying the following conditions:
\begin{itemize}
\item Each row vector $\bm{x}_i$ belongs to the same irreducible subset with $\bm{v}'_i$ and particularly $\bm{x}_2$ belongs to a fixed complete set of representatives for $S_{7,7}/\varGamma$, where $|S_{7,7}/\varGamma|=4$ by computing. 
\item $G(\bm{x}_1,\dots,\bm{x}_{40})=G(\bm{v}'_1,\dots,\bm{v}'_{40})$.
\end{itemize}
Exploring on a computer, it is shown that there are exactly $2$ matrices $X_1$,$X_2$ with these properties. Since $\varGamma$ acts transitively on each $\varGamma$-orbit of $S_{7,7}$, for every $g\in{\rm Aut}(O_{40})$ there is an element $\gamma$ of $\varGamma$ such that $B'g\gamma V=X_1$ or $B'g\gamma V=X_2$. Hence ${\rm Aut}(O_{40})$ is obtained by adding $B'^{-1}X_1V^{-1}$ and $B'^{-1}X_2V^{-1}$ to $\varGamma$. As a result, one has
\begin{equation*}
|{\rm Aut}(O_{40})|=684=2|\varGamma|.
\end{equation*}
By direct calculation, one can verify that ${\rm Aut}(O_{40})$ has the generators $g_1$,$g_2$ of orders $36$,$19$ respectively such that $g_1g_2g_1^{-1}=g_2^3$ (cf. Section $6$). Hence ${\rm Aut}(O_{40})$ contains the cyclic group $\langle g_2\rangle$ of order $19$ as a normal subgroup. Further one has
\begin{equation*}
\langle g_2\rangle\langle g_1\rangle={\rm Aut}(O_{40}),\ \ \langle g_2\rangle\cap\langle g_1\rangle=\{I_{40}\}.
\end{equation*}
Therefore ${\rm Aut}(O_{40})$ is isomorphic to the semidirect product of $\langle g_2\rangle$ and $\langle g_1\rangle$.
Regarding $\mathbb{Z}/19\mathbb{Z}$ and $\mathbb{Z}/36\mathbb{Z}$ as additive groups, this implies that there is the following isomorphism:
\begin{equation*}
{\rm Aut}(O_{40})\ni g_2^ag_1^b\mapsto(a,b)\in\mathbb{Z}/19\mathbb{Z}\rtimes_{\psi}\mathbb{Z}/36\mathbb{Z},
\end{equation*}
where $\psi$ is the homomorphism given in Theorem \ref{t2}, and thus the proof has been completed.

\section{Appendix}
In Tables $6.1$-$6.3$ we give examples of a Gram matrix $G(\bm{v}'_1,\dots,\bm{v}'_{40})$ of $O_{40}$ and the generators $g_1$, $g_2$ of ${\rm Aut}(O_{40})$ with respect to the basis $\{\bm{v}'_1,\dots,\bm{v}'_{40}\}$, which satisfy the conditions described in the previous section:

\setcounter{MaxMatrixCols}{40}
\renewcommand{\arraystretch}{0.7}
{\scriptsize
\begin{table}[H]
\begin{equation*}
{\tabcolsep=0.3mm
\left(
\begin{tabular}
{cccccccccccccccccccccccccccccccccccccccc}
4&-1&-2&2&0&0&1&0&1&1&1&1&0&0&1&1&0&0&1&-1&1&1&0&-1&0&0&0&0&0&0&1&-1&0&
      0&-1&1&0&0&-1&-1\\-1&4&0&0&0&-1&0&-1&0&-1&-1&-1&0&1&0&0&0&1&1&0&-1&0&-1&0&0&1&0&0&
      0&-1&-1&1&0&0&0&-1&-1&0&0&-1\\-2&0&4&-2&0&2&-1&0&0&0&-1&-1&0&0&0&-1&0&0&-1&0&0&
      -1&0&0&0&0&0&1&0&0&1&1&0&0&2&0&0&0&2&2\\2&0&-2&4&0&0&0&0&1&0&0&1&0&0&1&1&0&0&0&0&0&0&0&-1&0&0&0&0&0&
      0&-1&0&0&0&-1&1&0&0&-1&-1\\0&0&0&0&4&1&0&1&0&1&0&-1&0&1&0&-1&-1&0&-1&1&-1&-1&1&0&
      1&0&0&-1&0&1&0&0&-1&0&0&-1&-1&1&0&0\\0&-1&2&0&1&4&-2&2&0&0&-1&-1&0&0&0&-1&0&0&-2&1&0&-2&2&-1&0&-2&2&0&0&2&
      1&1&0&0&1&-1&0&0&1&1\\1&0&-1&0&0&-2&4&-1&1&1&1&1&-1&-1&1&1&0&1&1&-1&1&1&-1&0&0&1&
      -2&1&-1&-2&0&-1&1&0&0&1&0&0&0&0\\0&-1&0&0&1&2&-1&4&0&0&-1&0&0&0&-1&0&0&0&-2&
      1&-1&-2&2&-1&0&-2&2&-1&0&2&-1&0&0&0&0&-2&0&0&0&0\\1&0&0&1&0&0&1&0&4&0&1&1&1&-1&1&1&2&1&0&-1&1&1&0&-1&-2&-1&-1&0&-2&0&0&-1&
      2&-1&0&1&1&-1&0&0\\1&-1&0&0&1&0&1&0&0&4&0&1&1&1&1&0&-1&0&0&0&0&1&-1&-1&1&1&-1&0&
      0&-1&1&1&-1&1&0&1&-1&2&0&0\\1&-1&-1&0&0&-1&1&-1&1&0&4&1&-1&-2&1&1&2&0&0&-1&1&
      1&0&0&-1&-1&-1&0&-2&0&1&-1&1&-1&0&1&2&-1&0&0\\1&-1&-1&1&-1&-1&1&0&1&1&1&4&0&-1&1&1&1&0&0&0&1&1&0&-1&0&0&0&0&0&0&0&0&1&
      0&0&1&0&1&0&0\\0&0&0&0&0&0&-1&0&1&1&-1&0&4&2&-1&0&0&-1&1&0&0&1&0&0&0&0&0&0&0&0&
      0&0&0&0&-1&1&0&0&-1&-1\\0&1&0&0&1&0&-1&0&-1&1&-2&-1&2&4&-1&-1&-2&-1&1&0&-2&0&0&
      1&2&2&0&0&2&0&0&1&-2&1&-1&0&-2&2&-1&-1\\1&0&0&1&0&0&1&-1&1&1&1&1&-1&-1&4&1&0&0&0&0&1&1&0&-1&-1&0&0&0&-1&-1&1&
      0&0&-1&0&1&0&1&0&0\\1&0&-1&1&-1&-1&1&0&1&0&1&1&0&-1&1&4&1&-1&0&-1&1&1&-1&0&-1&-1&
      0&0&-1&-1&0&0&2&0&-1&1&2&-1&-1&-1\\0&0&0&0&-1&0&0&0&2&-1&2&1&0&-2&0&1&4&0&0&-1&1&0&0&-1&-2&-2&0&0&-2&0&
      0&0&2&-1&0&0&2&-2&0&0\\0&1&0&0&0&0&1&0&1&0&0&0&-1&-1&0&-1&0&4&0&1&0&0&0&-2&0&0&
      0&0&-1&0&0&0&0&-1&1&-1&-1&0&1&0\\1&1&-1&0&-1&-2&1&-2&0&0&0&0&1&1&0&0&0&0&4&-1&0&1&-1&0&1&2&-1&0&1&-2&
      0&-1&-1&0&-1&1&-1&0&-1&-1\\-1&0&0&0&1&1&-1&1&-1&0&-1&0&0&0&0&-1&-1&1&-1&4&-1&-1&1&
      -1&1&-1&2&-1&0&2&-1&1&-1&0&0&-1&-1&1&0&0\\1&-1&0&0&-1&0&1&-1&1&0&1&1&0&-2&1&1&1&0&0&-1&4&1&0&0&-2&-1&0&1&-1&-1&1&
      -1&2&0&0&1&2&-1&0&0\\1&0&-1&0&-1&-2&1&-2&1&1&1&1&1&0&1&1&0&0&1&-1&1&4&-2&0&-1&1&
      -1&0&-1&-1&1&0&1&0&-1&1&0&0&-1&-1\\0&-1&0&0&1&2&-1&2&0&-1&0&0&0&0&0&-1&0&0&-1&
      1&0&-2&4&0&0&-2&2&-1&0&2&0&-1&0&-1&0&-1&0&0&0&0\\-1&0&0&-1&0&-1&0&-1&-1&-1&0&-1&0&1&-1&0&-1&-2&0&-1&0&0&0&4&0&1&-1&1&1&0&0&-1&0&
      1&0&1&1&0&0&0\\0&0&0&0&1&0&0&0&-2&1&-1&0&0&2&-1&-1&-2&0&1&1&-2&-1&0&0&4&2&0&0&2&
      0&0&1&-2&1&0&0&-2&2&0&0\\0&1&0&0&0&-2&1&-2&-1&1&-1&0&0&2&0&-1&-2&0&2&-1&-1&1&
      -2&1&2&4&-2&1&2&-2&0&0&-2&1&0&1&-2&2&0&0\\0&0&0&0&0&2&-2&2&-1&-1&-1&0&0&0&0&0&0&0&-1&2&0&-1&2&-1&0&-2&4&-1&0&2&0&1&
      0&0&0&-2&0&0&-1&-1\\0&0&1&0&-1&0&1&-1&0&0&0&0&0&0&0&0&0&0&0&-1&1&0&-1&1&0&1&-1&4&
      0&-1&1&0&0&0&1&2&1&0&1&0\\0&0&0&0&0&0&-1&0&-2&0&-2&0&0&2&-1&-1&-2&-1&1&0&-1&
      -1&0&1&2&2&0&0&4&0&0&0&-2&1&0&0&-2&2&0&0\\0&-1&0&0&1&2&-2&2&0&-1&0&0&0&0&-1&-1&0&0&-2&2&-1&-1&2&0&0&-2&2&-1&0&4&0&0&
      0&0&0&-1&0&0&0&0\\1&-1&1&-1&0&1&0&-1&0&1&1&0&0&0&1&0&0&0&0&-1&1&1&0&0&0&0&0&1&0&
      0&4&0&0&-1&1&1&0&0&1&0\\-1&1&1&0&0&1&-1&0&-1&1&-1&0&0&1&0&0&0&0&-1&1&-1&0&
      -1&-1&1&0&1&0&0&0&0&4&0&1&0&-1&-1&1&0&0\\0&0&0&0&-1&0&1&0&2&-1&1&1&0&-2&0&2&2&0&-1&-1&2&1&0&0&-2&-2&0&0&-2&0&0&0&
      4&0&0&0&2&-2&0&0\\0&0&0&0&0&0&0&0&-1&1&-1&0&0&1&-1&0&-1&-1&0&0&0&0&-1&1&1&1&0&0&
      1&0&-1&1&0&4&-1&0&0&1&-1&0\\-1&0&2&-1&0&1&0&0&0&0&0&0&-1&-1&0&-1&0&1&-1&0&0&
      -1&0&0&0&0&0&1&0&0&1&0&0&-1&4&0&0&0&2&1\\1&-1&0&1&-1&-1&1&-2&1&1&1&1&1&0&1&1&0&-1&1&-1&1&1&-1&1&0&1&-2&2&0&-1&1&
      -1&0&0&0&4&1&0&0&0\\0&-1&0&0&-1&0&0&0&1&-1&2&0&0&-2&0&2&2&-1&-1&-1&2&0&0&1&-2&-2&
      0&1&-2&0&0&-1&2&0&0&1&4&-2&0&0\\0&0&0&0&1&0&0&0&-1&2&-1&1&0&2&1&-1&-2&0&0&1&-1&0&0&0&2&2&0&0&2&0&
      0&1&-2&1&0&0&-2&4&0&0\\-1&0&2&-1&0&1&0&0&0&0&0&0&-1&-1&0&-1&0&1&-1&0&0&-1&0&0&0&
      0&-1&1&0&0&1&0&0&-1&2&0&0&0&4&2\\-1&-1&2&-1&0&1&0&0&0&0&0&0&-1&-1&0&-1&0&0&-1&0&0&-1&0&0&0&0&-1&0&0&0&
      0&0&0&0&1&0&0&0&2&4
\end{tabular}
\right).
}
\end{equation*}
\vspace{-3mm}
\caption{A Gram matrix of $O_{40}$.}
\vspace{-3mm}
\end{table}
}


\setcounter{MaxMatrixCols}{40}
\renewcommand{\arraystretch}{0.7}
{\scriptsize
\begin{table}
\begin{equation*}
{\tabcolsep=0.25mm
\left(
\begin{tabular}
{cccccccccccccccccccccccccccccccccccccccc}
1&3&4&1&-1&2&1&-3&5&4&0&3&-5&-5&-9&-4&-7&-5&2&1&-4&-2&-1&0&-3&-1&3&-1&-1&-3&2&0&1&
      -5&-4&0&4&4&-2&-2\\-5&0&-5&-3&-4&4&1&-8&9&8&4&-2&-5&5&-2&3&-1&-5&-5&0&3&-1&0&-3&3&4&
      7&2&3&2&-5&-4&-2&-2&-1&-3&-3&-6&3&-1\\-2&-1&-2&-1&0&1&0&-1&1&-1&1&0&0&2&1&1&0&0&-1&
      -1&1&0&-1&-1&1&-1&1&1&0&0&-1&-1&-1&1&0&0&-2&-1&1&0\\0&1&1&0&-2&2&1&-3&4&4&1&0&-3&-1&-4&-1&-2&-3&-1&1&-1&0&0&0&0&1&3&0&1&-1&-1&
      -1&0&-3&-2&-1&1&0&0&-1\\4&-1&2&2&4&-3&-1&10&-12&-10&-4&1&7&0&8&0&6&8&4&0&1&3&0&4&0&
      -3&-10&-2&-2&1&3&4&0&4&4&2&1&2&-2&2\\3&-2&2&1&3&-3&-2&6&-8&-7&-2&1&5&0&6&0&3&6&3&-1&0&1&-1&3&0&-4&-6&0&-2&0&1&
      2&1&3&2&1&-1&2&-1&1\\-2&2&0&0&-2&2&2&-5&6&6&1&0&-4&-2&-6&-1&-4&-5&-1&0&-1&-2&0&-2&
      -1&3&5&0&1&1&0&-1&0&-4&-3&-1&2&0&0&-1\\2&-2&2&2&2&-4&-2&5&-6&-4&-2&0&3&0&4&0&2&4&4&0&0&0&0&2&-1&-2&-5&0&-2&1&2&
      2&2&2&2&0&0&2&-1&1\\2&-1&2&1&2&-2&0&3&-4&-4&-2&0&3&-2&2&0&2&3&1&-1&0&1&0&2&0&-2&
      -3&-1&-1&0&1&2&-1&1&1&1&0&2&-1&0\\2&1&2&1&1&0&1&4&-3&-4&-1&2&2&-2&0&-2&-1&1&2&
      1&-1&0&0&1&-2&0&-4&-2&-1&-1&3&3&0&0&0&1&3&2&-2&0\\6&1&6&4&5&-6&-1&11&-11&-10&-5&1&6&-6&3&-3&3&6&5&0&-2&3&1&4&-2&-5&-9&-3&-3&-2&6&5&1&
      3&2&4&3&7&-3&1\\-1&2&1&0&-2&3&2&-5&7&5&1&1&-5&-4&-9&-2&-6&-6&0&1&-3&-2&1&-2&-3&3&5&
      -1&0&-2&1&0&0&-4&-4&0&3&2&0&-2\\2&-2&-1&-1&2&-1&-1&5&-7&-7&-2&0&5&2&7&2&5&5&0&0&
      2&2&1&1&1&0&-7&-1&-1&1&0&2&-1&4&4&1&-1&-1&-1&1\\-1&0&-3&-2&-1&3&0&0&0&-1&0&1&0&3&2&2&1&0&-1&1&2&1&1&-1&1&2&-3&0&0&1&-1&0&
      -2&1&2&-1&0&-3&0&0\\-2&3&1&0&-3&3&2&-6&10&7&2&1&-6&-4&-9&-3&-7&-7&-1&0&-2&-2&-1&-1&
      0&0&7&0&1&-2&0&-1&-1&-5&-5&-1&2&1&0&-2\\-10&5&-4&-3&-11&10&4&-23&30&25&9&-2&-19&-2&-23&-1&-16&-22&-6&2&-2&-7&0&-9&-1&8&22&4&4&-1&
      -5&-8&-1&-12&-10&-6&0&-3&3&-5\\3&-2&2&1&3&-4&-1&5&-7&-6&-2&-1&5&-1&5&0&4&5&1&-1&0&1&
      0&2&0&-2&-4&-1&-1&0&1&2&1&3&2&2&-1&2&0&0\\1&-1&1&1&2&-2&0&4&-5&-3&-1&0&3&1&4&0&3&4&1&-1&1&1&-1&2&0&-1&-2&-1&0&1&
      0&1&0&1&1&1&0&0&0&1\\-3&2&-2&-2&-3&5&2&-8&9&7&2&1&-6&0&-7&-1&-5&-7&-3&0&-1&-2&-1&
      -3&0&2&7&0&1&-1&-2&-3&-1&-4&-3&-1&1&-1&1&-2\\-1&-1&-1&0&0&1&1&2&-3&-1&1&-1&2&3&4&1&3&3&0&-1&2&1&-1&1&1&1&0&-1&1&2&-2&0&
      0&1&1&1&-1&-2&1&1\\-1&1&-1&-2&-1&2&0&-5&6&5&1&0&-4&-2&-5&-1&-4&-4&-1&0&-2&-2&0&-2&
      -1&2&5&1&1&-1&-1&-1&1&-3&-3&0&1&0&0&-1\\0&2&1&0&-1&2&1&-2&4&2&1&1&-2&-3&-5&-1&-3&-3&-1&1&-1&-1&0&0&-1&1&2&-1&0&-2&1&0&
      -1&-2&-2&0&1&1&-1&-1\\5&-2&3&2&5&-6&-3&10&-13&-10&-6&1&7&-1&9&0&6&9&5&-1&0&3&1&4&0&
      -5&-11&-1&-3&1&3&4&1&5&4&2&1&3&-2&2\\0&0&-2&-1&0&-1&-1&0&-1&-1&-1&-1&1&1&2&2&2&1&
      -1&0&1&1&2&-1&1&0&-2&1&0&1&0&0&-1&2&2&0&-1&-1&0&1\\-3&2&-2&-1&-3&5&2&-5&6&5&2&1&-5&1&-6&0&-4&-6&-1&1&0&-1&0&-3&-1&3&4&0&0&0&-1&
      -2&-1&-3&-2&-1&1&-1&1&-1\\-3&2&-3&-2&-3&5&2&-6&8&5&2&1&-5&1&-6&0&-4&-6&-3&1&0&-1&0&
      -3&0&3&5&0&1&-1&-1&-2&-2&-3&-2&-1&1&-2&1&-1\\-2&0&-1&-1&-2&2&-1&-4&5&5&2&0&-4&1&-3&0&-3&-3&0&0&0&-2&-1&-1&0&1&4&1&0&0&-2&-2&1&
      -2&-2&-1&-1&-1&1&-1\\0&0&-1&-1&0&0&0&0&0&-1&0&0&0&0&0&0&0&0&-1&0&0&0&0&-1&0&1&0&0&
      0&0&0&0&0&0&0&1&0&0&0&0\\-5&2&-4&-3&-5&7&2&-12&14&11&4&1&-9&1&-10&0&-8&-10&-3&1&-2&
      -3&0&-5&-1&4&10&2&2&-1&-3&-4&-1&-5&-4&-3&1&-3&2&-2\\5&-2&4&3&5&-5&-2&11&-14&-11&-4&1&8&-1&8&0&6&10&5&-1&0&3&0&5&-1&-5&-10&-2&-3&0&3&4&1&
      5&4&3&0&4&-2&2\\2&2&3&1&1&0&0&1&0&-2&-1&3&0&-5&-4&-3&-4&-1&2&0&-3&-1&-1&1&-2&-3&-1&
      -1&-2&-3&3&1&0&-1&-2&1&2&4&-2&-1\\-3&0&-3&-2&-3&4&1&-4&5&4&4&-1&-2&3&-1&2&-1&-3&-3&
      1&2&-1&0&-2&1&4&3&1&2&1&-3&-2&-1&-1&-1&-2&-2&-4&2&-1\\-2&0&-1&-1&-2&1&0&-6&6&6&2&-2&-3&-1&-4&2&-2&-4&-2&0&0&-2&1&-2&0&3&5&2&1&1&-2&-2&0&
      -2&-2&-2&-2&-1&1&-1\\-3&1&-3&-2&-3&5&1&-6&7&6&3&0&-5&2&-5&1&-4&-5&-2&1&0&-2&0&-3&-1&
      4&5&1&1&0&-2&-2&0&-3&-2&-1&0&-2&1&-1\\1&-1&1&1&2&-3&-1&3&-4&-4&-1&0&3&0&3&0&2&3&1&-1&0&1&0&1&0&-2&-2&0&-1&0&1&1&
      0&2&1&1&-1&1&0&1\\0&2&1&0&-1&2&2&-2&3&1&0&1&-2&-3&-5&-2&-3&-3&-1&0&-2&0&0&-1&-1&0&
      2&-1&0&-2&1&0&-1&-2&-2&1&2&2&-1&-1\\0&0&0&0&0&-2&-1&0&1&1&0&-2&-1&-1&-1&0&0&-1&0&
      0&0&0&1&-1&0&0&1&1&0&0&0&0&1&0&0&1&-1&1&0&0\\-3&2&-3&-2&-3&5&2&-5&8&5&2&1&-5&0&-6&0&-5&-6&-1&1&0&-1&1&-3&-1&4&4&0&1&0&-1&
      0&-2&-3&-2&-1&2&-2&1&-1\\1&-1&1&1&2&-3&0&3&-4&-4&-1&0&3&0&3&0&2&3&1&-1&0&1&0&1&0&
      -3&-2&0&0&0&1&1&0&2&1&1&-1&1&0&1\\1&-1&1&1&2&-2&0&3&-4&-4&-1&0&3&0&3&0&2&3&1&-1&0&1&0&1&0&-2&-2&0&0&0&
      1&1&0&2&1&1&0&1&0&1
\end{tabular}
\right).
}
\end{equation*}
\vspace{3mm}
\caption{A generator $g_1$ of ${\rm Aut}(O_{40})$.}
\end{table}
}

\setcounter{MaxMatrixCols}{40}
\renewcommand{\arraystretch}{0.7}
{\scriptsize
\begin{table}[H]
\begin{equation*}
{\tabcolsep=0.30mm
\left(
\begin{tabular}
{cccccccccccccccccccccccccccccccccccccccc}
1&0&0&0&0&0&0&0&0&0&0&0&0&0&0&0&0&0&0&0&0&0&0&0&0&0&0&0&0&0&0&0&0&
      0&0&0&0&0&0&0\\-2&1&-1&-1&-2&2&1&-6&7&6&2&0&-4&-2&-6&-1&-5&-5&-1&0&-2&-2&0&-2&-1&1&
      6&1&1&-1&-1&-1&0&-3&-3&-1&1&0&1&-1\\0&0&1&0&0&-1&0&1&0&0&0&0&0&0&0&0&0&0&0&0&0&0&0&0&0&0&0&0&0&0&0&0&0&0&0&0&0&0&0&0\\0&0&-1&0&0&1&0&0&0&0&0&0&0&0&0&0&0&0&0&0&0&0&0&0&0&0&-1&0&0&0&0&0&0&0&0&0&0&0&0&0\\-3&1&-2&-1&-2&3&1&-5&6&6&2&0&-5&1&-5&0&-4&-5&0&1&-1&-2&0&-3&-2&3&5&1&1&0&-1&-2&1&
      -3&-2&-1&1&-1&1&-1\\0&0&0&0&0&0&0&1&0&0&0&0&0&0&0&0&0&0&0&0&0&0&0&0&-1&1&0&
      0&0&0&0&0&0&0&0&0&0&0&0&0\\0&0&0&0&0&0&0&0&0&0&0&0&0&0&1&0&0&0&0&0&0&0&
      0&0&1&0&0&0&0&0&0&0&0&0&0&0&0&-1&0&0\\0&0&0&0&0&0&0&0&0&0&0&0&-1&1&0&0&0&0&0&0&0&0&0&0&-1&0&0&0&0&0&0&0&0&0&0&0&0&0&0&0\\0&0&0&0&0&1&1&0&0&0&0&0&0&0&0&0&0&0&0&0&0&0&0&0&0&0&0&0&0&0&0&0&0&
      0&0&0&0&0&0&0\\0&0&0&0&0&0&0&0&0&0&0&0&0&0&0&1&0&0&0&0&0&0&1&0&0&1&0&0&0&0&0&
      0&0&0&0&0&0&0&0&0\\0&0&0&0&0&1&0&0&0&0&0&0&0&0&0&0&0&0&1&0&0&0&0&0&
      0&0&0&0&0&0&0&0&0&0&0&0&0&0&0&0\\0&0&0&0&0&0&0&-1&1&0&0&0&-1&0&-1&0&-1&-1&0&0&0&
      0&1&-1&0&0&0&0&0&0&0&0&0&0&0&0&0&0&0&0\\0&0&0&0&0&0&0&0&0&0&0&0&0&0&0&0&0&0&0&0&0&0&0&0&0&0&0&0&0&0&0&0&1&0&0&0&0&1&0&0\\0&0&0&0&0&-1&0&0&0&0&0&0&0&0&0&0&0&0&0&0&0&0&0&0&0&-1&0&0&0&0&0&0&0&0&0&0&0&1&0&0\\0&0&0&0&0&0&0&0&0&1&0&0&0&0&0&0&0&0&0&0&0&0&1&0&0&1&0&0&0&0&0&0&0&
      0&0&0&0&-1&0&0\\0&0&0&0&0&1&0&0&0&0&1&0&0&0&0&0&0&0&0&0&0&0&0&0&0&1&0&0&0&0&0&
      0&0&0&0&0&0&0&0&0\\0&0&0&0&0&1&0&0&0&0&0&0&0&0&0&0&0&0&0&0&0&0&0&0&
      0&0&0&0&0&0&0&0&0&0&0&0&0&0&0&0\\0&0&0&0&0&0&1&0&0&-1&0&0&1&-1&0&0&-1&0&0&-1&0&
      0&0&0&0&-1&0&0&0&0&0&1&-1&0&0&0&0&0&0&0\\0&0&0&0&0&0&0&0&0&0&0&0&0&0&0&0&0&0&0&0&0&1&0&0&1&-1&0&0&0&0&0&0&0&0&0&0&0&0&0&0\\-2&0&-2&-1&-1&1&0&-3&3&3&1&-1&-2&2&-1&1&-1&-2&-1&0&1&0&1&-2&0&2&2&1&1&1&-2&-1&0&-1&
      0&-1&0&-2&1&0\\0&0&0&0&0&1&0&0&0&0&0&1&0&0&0&0&0&0&0&0&0&0&0&0&0&1&0&0&
      0&0&0&0&0&0&0&0&0&-1&0&0\\0&0&0&0&0&0&0&0&0&0&0&0&0&0&0&0&0&0&0&0&1&0&0&0&1&
      0&0&0&0&0&0&0&0&0&0&0&0&0&0&0\\0&0&0&0&0&0&0&0&0&0&0&0&0&0&0&0&0&0&0&0&0&0&0&0&-1&0&0&0&0&0&0&0&0&0&0&0&-1&0&0&0\\0&0&0&0&0&0&0&0&0&0&0&0&0&0&0&0&1&0&0&1&0&0&0&0&0&1&0&0&0&0&0&0&0&
      0&0&0&0&0&0&0\\2&0&1&1&2&-3&-1&5&-5&-5&-2&0&3&-1&3&0&2&3&2&0&0&2&1&1&0&-2&-5&-1&
      -1&0&2&2&0&2&2&1&1&2&-1&1\\0&0&0&0&0&-1&0&0&0&0&0&0&0&0&0&0&0&0&0&0&0&
      0&0&0&1&-1&0&0&0&0&0&0&0&0&0&0&0&0&0&0\\0&0&0&0&0&-1&-1&-1&1&1&0&0&-1&0&-1&0&-1&-1&0&0&0&0&1&-1&-1&0&0&1&0&0&0&0&0&
      0&0&-1&0&0&0&0\\0&0&0&0&0&-1&0&1&1&0&0&0&0&0&1&-1&0&0&0&0&0&0&0&0&1&0&0&0&1&0&
      0&1&0&0&0&0&1&-1&0&0\\0&0&0&0&0&-1&0&0&0&0&0&0&0&0&0&0&0&0&0&0&0&0&0&
      0&0&0&0&0&0&1&0&0&0&0&0&0&0&0&0&0\\-5&2&-3&-2&-5&5&2&-11&14&11&4&-1&-9&1&-10&0&-7&
      -10&-3&1&0&-3&0&-5&-1&4&10&2&2&0&-3&-4&-1&-5&-4&-3&0&-2&2&-2\\0&0&-1&-1&0&0&0&1&0&-1&0&0&1&0&1&1&0&0&0&0&1&0&1&-1&0&2&-1&0&0&1&0&1&-1&
      1&1&0&0&-1&0&0\\1&0&1&1&1&-2&-1&3&-3&-2&-1&0&2&-1&2&0&1&2&1&0&0&1&1&1&0&-1&-3&0&
      0&0&1&1&0&1&1&0&1&1&-1&1\\0&0&0&0&0&1&0&0&0&0&0&0&0&-1&0&0&0&0&0&0&0&
      0&0&0&0&1&0&0&0&0&0&0&0&0&0&0&0&0&0&0\\-4&1&-2&-1&-3&3&1&-8&9&8&3&-1&-6&1&-6&1&-3&-6&-3&0&0&-1&0&-3&1&2&8&2&2&0&-3&-4&-1&
      -3&-3&-2&-1&-2&2&-1\\5&-2&3&1&4&-5&-2&8&-10&-9&-4&1&7&-2&7&0&4&7&3&0&-1&1&1&3&-1&-2&
      -9&-1&-2&0&3&4&1&4&3&1&1&2&-2&1\\0&0&0&0&0&0&0&0&1&0&0&0&0&0&0&0&0&0&0&0&0&0&0&0&1&0&0&0&0&0&0&0&0&0&0&0&0&0&0&0\\0&0&0&0&0&1&0&0&0&0&0&0&0&0&0&0&1&0&0&0&0&0&0&0&0&1&0&0&0&0&0&0&0&0&0&0&0&0&0&0\\0&0&0&0&0&-1&0&0&0&0&0&0&0&0&0&0&0&0&0&0&0&0&1&0&0&0&0&0&0&0&0&0&0&0&0&0&0&0&0&0\\1&-1&-1&-1&1&-1&0&4&-4&-4&-1&-1&3&1&4&1&3&3&0&0&1&1&1&0&0&2&-3&-1&0&1&0&2&0&2&
      2&2&0&-1&0&1\\1&-1&0&0&1&-2&-1&3&-3&-3&-1&-1&2&1&3&1&3&2&0&0&1&1&0&1&1&-1&-2&0&-1&
      0&0&0&0&2&2&1&-2&0&0&1
\end{tabular}
\right).
}
\end{equation*}
\vspace{-3mm}
\caption{A generator $g_2$ of ${\rm Aut}(O_{40})$.}
\end{table}
}

\end{document}